\documentclass[12pt,reqno]{article}

\usepackage{csvsimple}
\usepackage{color}
\usepackage{graphicx}
\usepackage[position=b]{subcaption}
\usepackage{braket}
\usepackage{amsthm}
\usepackage{tcolorbox}
\usepackage{amsmath,amsfonts,amssymb}
\usepackage{multirow}
\usepackage{adjustbox}
\usepackage{stmaryrd}

\newcommand{\half}{\frac 12}

\newcommand{\diff}{\;\textnormal{d}}

\newcommand{\ds}{\;\textnormal{d}s}
\newcommand{\dx}{\;\textnormal{d}x}

\newcommand{\dth}{\;\textnormal{d}{\theta}}

\newcommand{\ltwonorm}[1]{\| #1 \|_{0, \Omega}}

\newcommand{\Rd}{{\mathbb{R}^d}}

\newcommand{\idop}{\textnormal{\textbf{id}}}

\newcommand{\inv}{^{-1}}

\newcommand{\intcirc}{{\int_{S^1}}}
\newcommand{\linftysd}[1]{\|#1\|_{\infty, S^1}}

\newcommand{\intpk}{\int_{\partial K}}
\newcommand{\oin}{{\Omega^-}}
\newcommand{\oout}{{\Omega^+}}
\newcommand{\oinh}{{\Omega_{h}^-}}
\newcommand{\oouth}{{\Omega_{h}^+}}
\newcommand{\bfalpha}{{\boldsymbol{\alpha}}}
\newcommand{\gradalpha}{\partial^\bfalpha}
\newcommand{\sumks}{\sum_{K\in\Omegah}}
\newcommand{\hnorm}[1]{\| #1 \|_\mdh}

\newcommand{\ltwoe}[1]{\| #1 \|_{0,e}}

\newcommand{\ltwoK}[1]{\| #1 \|_{0,K}}
\newcommand{\honeK}[1]{\| #1 \|_{1,K}}
\newcommand{\pones}{\mathbb{P}_{\femp}}
\newcommand{\iones}{\mathbb{I}_{\femp}}
\newcommand{\Kep}{{K_e^+}}
\newcommand{\Kem}{{K_e^-}}

\newcommand{\CR}{\Pi^c_h}
\newcommand{\CRi}{\big[\Pi^c_h\wsf\big]^i}
\newcommand{\femp}{\mathsf{S}_h}
\newcommand{\femq}{\mathsf{S}_h}
\newcommand{\pone}{\mathcal{P}^1(\Omega_h)}
\newcommand{\Omegah}{\Omega_h}

\newcommand{\poneK}{\mathbb{I}^{1,d}_{K}}

\newcommand{\mdh}{{m, \Omegah}}

\newcommand{\uth}{u_{h}}

\newcommand{\ltwonormh}[1]{\| #1 \|_{0, \Omegah}}

\newcommand{\ltwonorms}[1]{\| #1 \|_{0, S^1}}

\newcommand{\normmdh}[1]{\| #1 \|_{m,\Omegah}}

\newcommand{\femspace}{{\tilde{V}^d_{h}}}

\newcommand{\femspacetensor}{{\tilde{V}^{d\times d}_{h}}}

\newcommand{\meshskelint}{\mathring{\mathcal{E}}_{h}}
\newcommand{\gam}{{g\circ S^1}}
\newcommand{\meshskelng}{\meshskelint\setminus\gam}

\newcommand{\intep}{\int_e}
\newcommand{\jump}[1]{\llbracket #1 \rrbracket}
\newcommand{\avg}[1]{\{ #1 \}}

\newcommand{\vsf}{\mathsf{v}}
\newcommand{\wsf}{\mathsf{w}_h}

\newcommand{\Fspace}{L^2(S^1)^d}
\newcommand{\Jindex}{\mathcal{J}_g}

\newcounter{mypbm}
\newcommand{\defproblem}[2]{\refstepcounter{mypbm}\begin{tcolorbox} \textbf{Problem \arabic{mypbm} (#1)}\\
 #2 \end{tcolorbox}}
 
\newcounter{mytheorem}
\newtheorem{theorem}[mytheorem]{Theorem}

\newtheorem{definition}[mytheorem]{Definition}
\newtheorem{corollary}[mytheorem]{Corollary}
\newtheorem{remark}[mytheorem]{Remark}
\newtheorem{lemma}[mytheorem]{Lemma}
\newtheorem{proposition}[mytheorem]{Proposition}

\usepackage[
backend=biber,
style=alphabetic,
sorting=ynt
]{biblatex}

\begin{document}

\title{A note on error analysis for a nonconforming discretisation of the tri-Helmholtz equation with singular data}
\author{{\sc Andreas Bock and Colin Cotter}}
\maketitle

\begin{abstract}
{We apply the nonconforming discretisation of \cite{wu2019nonconforming} to the tri-Helmholtz equation on the plane where the source term is a functional evaluating the test function on a one-dimensional mesh-aligned embedded curve. We present error analysis for the convergence of the discretisation and linear convergence as a function of mesh size is recovered almost everywhere away from the embedded curve which aligns with classic regularity theory.}\\

{Elliptic PDE; Nonconforming method; Finite element method.}
\end{abstract}

\section{Introduction}\label{introduction}

We study the vector tri-Helmholtz equation on the unit square:
\begin{equation}\label{eq:trilap}
B u = f\quad \text{on}\quad \Omega:=[0,1]^d,
\end{equation}
where $d=2$ subject to homogeneous Dirichlet conditions on $\partial \Omega$, where $B$ is the linear differential operator defined as follows by the $L^2$ inner product (denoted by $\langle\cdot,\cdot\rangle_{0,\Omega}$), for some $b>0$ for smooth functions $u$ and $v$ and $m=3$, where $\idop$ represents the identity operator:
\begin{align}\label{eq:Djs}
\langle Bu, v\rangle_{0,\Omega} = \sum_{i=1}^d\int_{\Omega}
\langle (\idop - b\Delta)^{m} u^i, v^i\rangle \diff x,
\end{align}
where $\langle\cdot,\cdot\rangle$ is the standard Euclidean inner product. $f$ is a functional defined in the following way, where for the periodic domain $S^1 = [0,2\pi]$, $g\in \mathsf{C}^0(S^1, \Omega)$ is a continuous embedded planar curve:
\begin{equation}\label{eq:fdef}
v \mapsto f(v) := \intcirc\tilde{f}\cdot v\circ g\dth,
\end{equation}
where $\tilde{f}$ is some function in $\Fspace$. To simplify the analysis further we assert that the curve $g$ satisfies the following assumption:
\begin{equation}\label{distg}
\text{dist}(\partial \Omega, \text{supp}(g))=:\epsilon_g > 0.
\end{equation}
In this paper we apply the nonconforming finite element discretisation of \cite{wu2019nonconforming} to the components of $u$ in equation \eqref{eq:trilap} and study the convergence properties of this scheme for source terms $f$ as defined in equation \eqref{eq:fdef}. We denote by $\Omega_h$ a shape-regular, quasi-uniform triangulation of the domain $\Omega$ \cite[Definition 3.3.11]{brennerscott} and we shall assume that the curve $g$ is aligned with its mesh skeleton.\\

The finite element method \cite{brennerscott} is a ubiquitous method for discretising partial differential equations (PDEs) to which solutions are commonly sought after in piecewise polynomial spaces. Such finite element functions are identified by a set of basis coefficients and a polynomial basis. Determining what constitutes an adequate basis is a hugely important challenge. Generally speaking, the choice of the local degrees of freedom of a finite
element affects the properties of the global approximation space. For example,
when $X_h$ is constructed from linear polynomials i.e. for $K \in \Omega_h$,
$X_h|_K = \mathcal{P}^1(K)$ and in addition, functions in $X_h$ are linear
interpolates of the nodal values at the vertices of $K$, then it is easy to see
that $X_h \subset \mathsf{C}^0(\Omega)$. A nonconforming linear space also
exists where the local degrees of freedom are chosen at the midpoint of the
edges of $\partial K$ see e.g. the construction in \cite[Section
10.3]{brennerscott}. This results in a \emph{locally continuous}, \emph{globally discontinuous} finite element and $X_h\not\subset X$. While the continuity of the former can be advantageous in light of Céa's lemma \cite[Lemma 2.28]{ErnGuermond2013} for easy access to a convergence proof for the method, it is sometimes inconvenient or even computationally infeasible to
design conforming finite element spaces e.g. when the continuous space $X$ is some higher-order Sobolev space. Generally speaking, for a $2m$\textsuperscript{th}-order PDE a conforming space must be a $\mathsf{C}^{m-1}$-conforming subspace\footnote{The weak form
equation would require $m$ weak derivatives which implies the existence of $m-1$ continuous derivatives.}. This poses a challenge for the design and
implementation of finite element software. For instance, automation of code generation for $\mathsf{C}^1(\Omega)$-conforming finite elements was only very recently solved \cite{kirby2018general,kirby2019code}. Nonconforming FEMs are therefore attractive in developing numerical schemes for higher-order PDEs and - as we shall see - adequate convergence can be recovered in certain cases.
\subsection{Nonconforming Finite Elements}

For a triangulation $\Omega_h$ we define the
mesh resolution, or mesh size, as $h = \max_{K\in \Omega_h} h_K$,
where $h_K = \textnormal{diam}(K)$, see \cite{brennerscott,ciarlet2002finite}. As we will be studying
convergence of finite element methods we introduce the notion
of \emph{shape-regularity} \cite[Definition 1.107]{ErnGuermond2013}. A family of meshes $\{ \Omega_h \}_{h>0}$ is said to be shape-regular if there exists $c_0>0$ independent of $h$ such that
\begin{equation}\label{eq:shapereg}
\Omega_h \frac{h_K}{\rho_K} \leq c_0,\quad \forall \Omega_h \in \{ \Omega_h \}_{h>0},
\end{equation}
where $\rho_K$ is the diameter of the largest circle inscribed in $K$. We also
assume that such families are \emph{quasi-uniform} \cite[Definition
1.140]{ErnGuermond2013} i.e.  there exists a constant $c>0$ such that:
\[
\forall h,\;\forall K\in \Omega_h,\quad h_K\geq ch.
\]
This implies that $\epsilon_g$ from \eqref{distg} satisfies $h<\epsilon_g$ i.e. the discretisation of the domain is so that the curve is at least $h$ away from the boundary. Given a triangulation $\Omega_h$ we define its
interior facets (or edges, when $d=2$) $\mathring{\mathcal{E}}_h$ by
 $e\in \mathring{\mathcal{E}}_h$ if $\exists K, K' \in \Omega_h$ such
 that  $e = K\cap K'$. We denote by $\partial \Omega_h$
the edges of $\partial K$, $K\in \Omega_h$ that trace $\partial \Omega$, and we say that $\mathcal{E}_h = \mathring{\mathcal{E}}_h \cup \partial \Omega_h$ is the \emph{mesh skeleton} of $\Omega_h$. Sometimes we explicitly denote the dependence of the skeleton on its associated triangulation i.e. $\mathcal{E}_h(\Omega_h)$.  We also use the notation $\int_{\mathcal{E}_h} = \sum_{e\in\mathcal{E}_h}  \int_e$ with
similar interpretations for $\int_{\partial \Omega_h}$ and
$\int_{\mathring{\mathcal{E}}_h}$ and we have:
\[
\int_{\mathcal{E}_h} = \int_{\mathring{\mathcal{E}}_h}+\int_{\partial \Omega_h}\,.
\]

The space $\mathsf{C}^k(O, \Rd)$ denotes the space of continuous functions over
the domain $O$ with $k$ continuous derivatives taking value in $\Rd$.  When the
latter is omitted, the range is $\mathbb{R}$ e.g.  $\mathsf{C}^0(O)$ is the
space of continuous functions over $O$ taking values in $\mathbb{R}$. $L^p(O)$, $p\geq 1$
and integer denotes the usual Lebesgue spaces over $O$ with $L^\infty(O)$ being
the Banach space of essentially bounded functions on $O$. For $k\geq 0$ (integer
or fraction) we let $W^{k,p}(O)$ define the usual Sobolev spaces, see
\cite{evans}.  Note that $W^{k,\infty}(O)$ is the Banach space of scalar
$k$-Lipschitz continuous functions\footnote{We remark that the Lipschitz class
of functions does not always coincide with $W^{k,\infty}(O)$ depending on the
nature of $O$.  Given that we are always dealing with convex domains we
characterise the Lipschitz class with the Sobolev spaces $W^{k,\infty}(O)$ (see
\cite[Theorem 4.1]{heinonen2005lectures}) and mention this no further.} over the
convex polygonal Lipschitz domain $O$ taking values in $\mathbb{R}$. We will be
dealing with vector and tensor-valued functions as well, and therefore when
necessary use superscripts e.g. $L^\infty(O)^d$ to denote vector-valued
functions $\mathbb{R}^d \rightarrow\mathbb{R}^d$, or $W^{k,\infty}(O)^{d\times
d}$ denotes the space of tensor-valued $k$-Lipschitz functions on $O$. This
should be understood in the sense that each component considered as a scalar
function satisfies the regularity of the given space.  We use standard
multi-index notation i.e. a $d$-dimensional multi-index $\bfalpha$, with
$|\bfalpha|=\sum_{i=1}^d\alpha_i$, and we define $\gradalpha$, the partial
derivative with respect to $\bfalpha$ as $\gradalpha =
{\frac{\partial^{|\bfalpha|}}{\partial^{\alpha_1}_{1} \cdots \partial^{\alpha_d}_{d}}}$, where $\partial^{\alpha_i}_i := \partial^{\alpha_i} / \partial {x_i}^{\alpha_i}$ . In this paper, $\diff{x}$ - sometimes also $\diff{\hat x}$ - denotes the
$d$-dimensional Lebesgue measure \cite[Section 1.1]{brennerscott}, where $d$
will be clear from context.\\

We make a clear distinction in our notation depending on the
value of $p$. This is reflected in the following norms when $p=2$ or $p=\infty$, respectively:
\begin{subequations}\label{ournormz}
\begin{align}
& \| f \|_{k,O}^2 = \sum_{|\bfalpha|\leq k} \| \gradalpha f\|_{0,O}^2,\\
& \| f \|_{W^{k,\infty}(O)} = \sum_{|\bfalpha|\leq k} \| \gradalpha
f\|_{\infty,O},
\end{align}
\end{subequations}
where $\|\gradalpha f\|_{0,O}^2 = \int_O |\gradalpha f|^2 \dx$ is the usual
$L^2$ norm over $O$ and $\| \gradalpha f\|_{\infty,O} = \textnormal{ess}\,
\sup_O |\gradalpha f|$ is the essential supremum norm. Note in particular that
due to the presence of the Euclidean norm $|\cdot|$ here, the norms in
\eqref{ournormz} are well-defined for scalar, vector and tensor-valued
objects. When $p=2$ we also use the notation $H^k(O):=W^{k,p}(O)$ with $H_0^k(O)$ consisting of such functions that vanish on $\partial O$ in sense of traces. We also denote by $|\cdot|_{k, O}$ the usual semi-norm over this space. The inner product on $H^k(O)$ is denoted by $\langle \cdot,\cdot \rangle_{k,O}$.\\
\begin{definition}[Local interpolant {\cite[Lemma 4.4.1 and Theorem
4.4.4]{brennerscott}}]\label{def:localinterpolant}
For a convex simplex $K\subset\mathbb{R}^d$ and nonnegative integers $l$ and
$m$ we define the usual local interpolant $\mathcal{I}_K : \mathsf{C}^l(\bar{K})
\rightarrow H^m(K)$ satisfying the following estimate:
\[
\|\mathcal{I}_K u\|_{m,K} \lesssim \|u\|_{\mathsf{C}^l(\bar{K})}.
\]
When $m-l-d/2>0$ we have the following bound:
\begin{align*}
& |u - \mathcal{I}_K u|_{i,K} \lesssim h_K^{m-i} |u|_{m, K}, \quad 0\leq
i\leq m,
\end{align*}
with a constant that depends on $m$, $d$ and $K$.

\end{definition}
This extends trivially to vector or tensor-valued objects. We can therefore,
under the right circumstances, approximate smooth functions on $K$ by a
polynomial. For nonsmooth functions one must seek other interpolants since 
the nodal basis may not support point evaluation i.e. when the space
$\mathsf{C}^l(\bar{K})$ in definition \ref{def:localinterpolant} is replaced by
a weaker space like $L^2(\bar{K})$. Local averages are often taken as surrogates
for the nodal basis in this case, see e.g. the Cl\'ement
\cite{clement1975approximation}, or for interpolation operator satisfying
boundary conditions see Scott-Zhang \cite{scott1990finite}.  For more
information on this see the discussion in \cite[Section 4.8]{brennerscott}, and
\cite{dupont1980polynomial} for general polynomial approximation of Sobolev
spaces. A central task in finite element theory is the construction of the nodal basis and the basis functions to cater for convergent approximations of some infinite-dimensional problem. We shall therefore leave these undefined here and return to their construction when necessary in the coming chapters.\\

From definition \ref{def:localinterpolant} we can define a global interpolant $\mathcal{I}_h: X \rightarrow X_h$ by:
\begin{equation}\label{eq:globint}
\mathcal{I}_h v|_K = \mathcal{I}_K v, \qquad K \in O_h,
\end{equation}
where $O_h$ is a triangulation of $O \subset \mathbb R^d$ over which the spaces
$X$ and $X_h$ are defined. 
Global interpolation estimates require the notion of
a \emph{broken} Hilbert-Sobolev norm given by the following definitions:
\begin{align}
\|u\|_{k,O_h}^2 := \sum_{K\in O_h} \|u\|_{k, K}^2,\label{brokenmd_scalar}
\end{align}
with the associated broken inner product:
\begin{align*}
\langle u, v\rangle_{k,O_h} := \sum_{K\in O_h} \langle u, v\rangle_{k, K}.
\end{align*}
To study finite element methods we introduce the notion of \emph{approximability}:
\begin{lemma}[Approximability]\label{lemma:approx}
Suppose $X=H^{m}(O)$ where $O$ is a Lipschitz domain in $\mathbb R^d$ and
let $\mathcal{I}_h$ be the global interpolant from \eqref{eq:globint}. Then for
$u \in H^{m+s}(O)$:
\begin{subequations}\label{eq:lemma:approx}
\begin{align}
& \|u - \mathcal{I}_K u\|_{i,O_h} \lesssim h_K^{m+s-i} |u|_{m+s, O_h}, \quad 0\leq
i\leq m,\label{eq:lemma:approx:a}\\
& \lim_{h\rightarrow 0} \inf_{v_h\in X_h} \|v-v_h\|_{X_h} = 0,\quad \forall
v\in X,\label{eq:lemma:approx:b}
\end{align}
\end{subequations}
and we sometimes say $X_h$ \emph{approximates} $X$.\\
\end{lemma}
\begin{proof}
\eqref{eq:lemma:approx:a} follows from the definition of $\mathcal{I}_h$ and
definition \ref{def:localinterpolant}.  \eqref{eq:lemma:approx:b} follows from
continuity of the interpolant and a standard density argument; see e.g. the proof of \cite[Theorem 2.1]{wang2013minimal}.
\end{proof}

Convergence for nonconforming methods are usually established in norms such as \eqref{brokenmd_scalar} owing to their local nature.
Approximation and consistency conditions for nonconforming finite
elements can now be formulated using more suitable vocabulary, see \cite{stummel1979generalized} or \cite[Section 3.1]{wang2013minimal} for the following definition.

\begin{definition}[Consistent approximation]\label{def:consist}
A finite element space $X_h$ over a bounded Lipschitz domain $\Omega$ is said to be a \emph{consistent approximation} of $X=H^m(\Omega)$, $m\geq 1$, if and only if:
\begin{enumerate}
\item $\lim_{h\rightarrow 0} \inf_{v_h\in X_h} \|v-v_h\|_{X_h} = 0,\quad \forall
v\in X$ (approximation).
\item For any sequence $\{ v_h \}_{h>0}$ with $v_h \in X_{h}$ and
$h\rightarrow 0$ such that $\{ \partial_{h}^\bfalpha
v_{h}\}_{h>0}$ is $L^2(\Omega)$-weakly convergent to some $v^\bfalpha \in X$ for all 
$|\bfalpha| \leq m$, it holds that $v_0 \in X$ and $v^\bfalpha
=\partial^\bfalpha v_0$ for all $|\bfalpha| \leq m$ (consistency).
\end{enumerate}
\end{definition}
The first condition can be derived from \eqref{eq:lemma:approx:b} of lemma
\ref{lemma:approx} via a canonical interpolant. The second condition can be interpreted as a kind of weak compactness condition. In practice we employ so-called \emph{patch tests} which are sufficient conditions for a consistency condition but motivated by the fact that they are much simpler to prove and sometimes even verify by numerical computation. These often amount to verifying whether integral moments of derivatives vanish on the mesh skeleton - a property called a \emph{weak continuity}. Historically the first attempt at providing an automated solution to this issue was in engineering by Irons et al. \cite{bazeley1965triangular} in 1965 but proved to be neither
necessary nor sufficient by Stummel \cite{stummel1980limitations}. The
generalised patch test (GPT) of Stummel \cite{stummel1979generalized} states necessary and sufficient conditions leading to convergence of nonconforming elements. Generally speaking, the design of a finite element leads to some weak continuity of which the GPT does not make use in its original form. However, using the form of the GPT as stated in \cite[Equation 4.4]{wang2001necessity}, $X_h$ passes the GPT if and only if the following condition is satisfied:
\begin{equation}\label{eq:gpt-cond:BG}
\lim_{h\rightarrow 0} \sup_{\substack{v_h\in X_h \\ \|v_h\|_{X_h}\leq 1}} |
T_{\bfalpha, i} (\psi, v_h) | = 0,\quad |\mathbf{\bfalpha}|<m,\quad 1\leq i\leq d,
\quad \forall \psi \in \mathsf{C}^\infty(\bar{\Omega}),
\end{equation}
where:
\[
T_{\bfalpha, i} (\psi, v_h) = \sum_{K\in \Omega_h} \int_K \big[\psi \frac{\partial}{\partial
x_i} \partial^\bfalpha v_h + \frac{\partial\psi}{\partial x_i} \partial^\bfalpha v_h 
\big]\diff x\,.
\]
Equivalently, where $\eta_K$ is the outward normal of $K\in \Omega_h$:
\begin{equation}
T_{\bfalpha, i} (\psi, v_h) = \sum_{K\in \Omega_h} \int_{\partial K} \psi
\partial^\bfalpha v_h \eta_K^i \diff s\,.
\end{equation}
Alternatives were also proposed, see for instance F-E-M test
\cite{shi1987fem} and the IPT test \cite{hong1986compactness}.  A major step forward was taken in \cite{wang2001necessity} where a conjecture was proved saying that given suitable approximation properties (in spirit of lemma \ref{lemma:approx}) and weak continuity, a \emph{weak patch test} (WPT) can be designed which is extremely simple and convenient. We refer the reader to \cite{shi2002nonconforming} for an excellent overview of the nonconforming FEM literature for second and fourth order PDEs.\\

As mentioned, we shall use a finite element space from \cite{wu2019nonconforming} to discretise a weak formulation of \eqref{eq:trilap} and we briefly summarise its construction. First we recall some notation. For a triangle $K$, let $\mathcal{P}_i(K)$ denote the space of $i^{\text{th}}$ order polynomials on $K$. Let $q_K$ denote the \emph{bubble function} on $K$ which is a nonnegative cubic polynomial that is zero on $\partial K$; see e.g. \cite{ErnGuermond2013} for more details. In the following, $\frac{\partial }{\partial\nu_e}$ denotes a derivative in the direction of the outward normal $\nu_e$ on an edge $e$ in the mesh skeleton of a triangulation. We define the following shape function space on a triangle $K\in\Omegah$:
\[
\tilde P_K^{(3,2)} := \mathcal{P}_3(K) + q_K \mathcal{P}_1(K) + q_K^2
\mathcal{P}_1(K).
\]
By \cite[Lemma 4.1]{wu2019nonconforming}, the following degrees of freedom
determine a function $v\in\tilde P_K^{(3,2)}$:
\begin{subequations}
\begin{align}
& \frac{1}{|e|}\int_e \frac{\partial^2 v}{\partial\nu_e^2}\ds & \text{$e$ is an
edge of $\partial K$},\\
& \nabla v(a) & \text{$a$ is a vertex of $\partial K$},\\
& \frac{1}{|e|}\int_e \frac{\partial v}{\partial\nu_e}\ds & \text{$e$ is an
edge of $\partial K$},\\
& v(a) & \text{$a$ is a vertex of $\partial K$}.
\end{align}
\end{subequations}
We now define the scalar finite element space $\tilde V_h$ defined over
$\Omegah$ as consisting of all functions $v_h|_K \in \tilde P_K^{(3,2)}$,
$K\in\Omegah$, with the degrees of freedom above being equal to zero if either
the edge or vertex lies on the boundary of $\Omegah$ (see also \cite[Equation
4.8]{wu2019nonconforming} for the definition of $\tilde V_h$):
This element is locally to $K$ a seventh order polynomial. It is designed in such a way that makes it possible to describe convergent $H^3$-nonconforming FEM. Our vector-valued finite element space $\femspace$ is then the space whose components occupy the
scalar finite element space $\tilde V_h$. As shown in the following corollary, this is a continuous finite element space.
\begin{corollary}
$\tilde V_h\hookrightarrow\mathsf{C}^0(\Omega)$.
\end{corollary}
\begin{proof}
Let $K$, $K'$ be two elements of $\Omegah$ that share an edge $e$. As shown in \cite{wu2019nonconforming}, there are 8 degrees of freedom on this edge. Let $v$ and $v'$ be the restriction of a function $w\in\tilde{V}_h$ to $\partial K$ and $\partial K'$, where these functions occupy $\mathcal{P}^3(\partial K)$ and $\mathcal{P}^3(\partial K')$, respectively. Values and derivatives at the vertices are continuous in $\tilde{V}_h$ in the sense that they are unique, so
$v$ and $v'$ are order 3 polynomials that agree on 4 values which is only
possible if $v=v'$ everywhere on $e$. Since $e$ was arbitrary, $\tilde{V}_h \hookrightarrow \mathsf{C}^0(\Omega)$. 
\end{proof}

Although the finite element space above is continuous, this does not extend to
higher-order derivatives. Nevertheless, the following central result
highlights its \emph{weak} continuity properties:
\begin{lemma}[Weak continuity {\cite[Lemma 4.2]{wu2019nonconforming}}]\label{lemma:wkcts}
Let $\Omegah$ denote the mesh and $e \in \meshskelint$ be an interior edge of a
triangle $K\in\Omegah$. For any component $u^i$, $i=1,2$ of $u\in\femspace$ and $K' \in
\Omegah$ such that $e\cap K'\neq\emptyset$,
\[
\int_e \partial^\bfalpha u^i|_K \diff{s} = \int_e \partial^\bfalpha
u^i|_{K'}\diff{s}, \quad |\bfalpha| = 1,2.
\]
When $e\in\partial\Omegah$,
\[
\int_e \partial^\bfalpha u^i|_K \diff{s} = 0.
\]
\end{lemma}

\begin{remark}
This importance of this result cannot be understated. Weak continuity is the
exact reason why we will be able to talk about convergence for nonconforming
methods. A na\"ive piecewise $H^3$ finite element space constructed from
piecewise cubic polynomials functions $\mathcal{P}^3(K)$, $K\in\Omegah$ does
not lead to the same results.
\end{remark}

Recalling $m=3$, we equip $\femspace$ with the broken $H^m$ norm $\normmdh{\cdot}$ dominating all derivatives up to third order. We shall also define a matrix-valued finite element space $\femspacetensor$ whose components occupy $\tilde V_h$, see the aforementioned reference.\\

We observe the following standard result.
\begin{theorem}\label{thm:H3Lip}
Let $O$ be a convex bounded Lipschitz domain in $\mathbb R^d$ with polygonal
boundary and $O_h$ a triangulation thereof satisfying the regularity
requirements introduced above. Suppose further that $u\in
\mathsf{C}(\bar{O})^d$, $u|_K \in H^3(K)^d$ for $K\in O_h$. Then $u\in
W^{1,\infty}(O)^d$.
\end{theorem}

\begin{proof}
The embedding theorem for homogeneous Sobolev spaces (i.e. with zero traces)
into the space $\mathsf{C}^j(\bar{O})$ is well-known. However, since the trace
$u|_K$ of $u$ on $\partial K$, $K\in O_h$ may not be zero we appeal to a
slightly different albeit standard result. By \cite[Theorem
5.4]{adams1975sobolev}, however, $H^m(K)\hookrightarrow\mathsf{C}_B^1(K)$,
where:
\[
\mathsf{C}_B^1(K) = \{ u \in \mathsf{C}^1(K)\;|\; D^\bfalpha u\textnormal{ is
bounded on } K,\; |\bfalpha|\leq 1\}. 
\]
This means any $H^m(K)$ function has a continuous representative with almost
everywhere bounded first derivatives on $K$. Since $u\in\mathsf{C}^0(\bar{O})$,
$u$ is a continuous function with its first derivative a.e. bounded, implying a
Lipschitz condition.
\end{proof}

\section{Tri-Helmholtz equation with singular data}\label{sec:trilap}

We define the bilinear form:
\begin{equation}\label{eq:ctsbilin}
a(u, v) =\sum_{i=1}^d\int_{\Omega} \sum_{j=0}^m b_j \binom{m}{j}\langle D^j u^i,
D^j v^i\rangle \diff x,
\end{equation}
where $D^0 = \idop$, and
\[
D^j =\begin{cases} \nabla D^{j-1} & j\text{ is odd},\\
\nabla\cdot D^{j-1} & j\text{ is even}, \end{cases}
\]
and $b_j = b^j$. We then see that by integrating by parts in \eqref{eq:trilap} we have the equivalence:
\[
a(u, v) = \langle Bu,v\rangle_{0,\Omega},
\]
for sufficiently smooth functions $u$ and $v$ with vanishing
$j$\textsuperscript{th} order derivatives on $\partial\Omega$, for
$j=0,\ldots,m-1$. In the rest of this paper the embedding $g$ will be a continuous piecewise linear map from $S^1$ into the plane which is aligned with the mesh skeleton similar to what is shown in figure \ref{fig:g}. Equipped with \eqref{eq:ctsbilin} we can pose the following weak version of \eqref{eq:trilap} where we write the 6\textsuperscript{th} order problem into a mixed system of two third order problems:

\begin{figure}
\centering
\includegraphics[width=.5\textwidth]{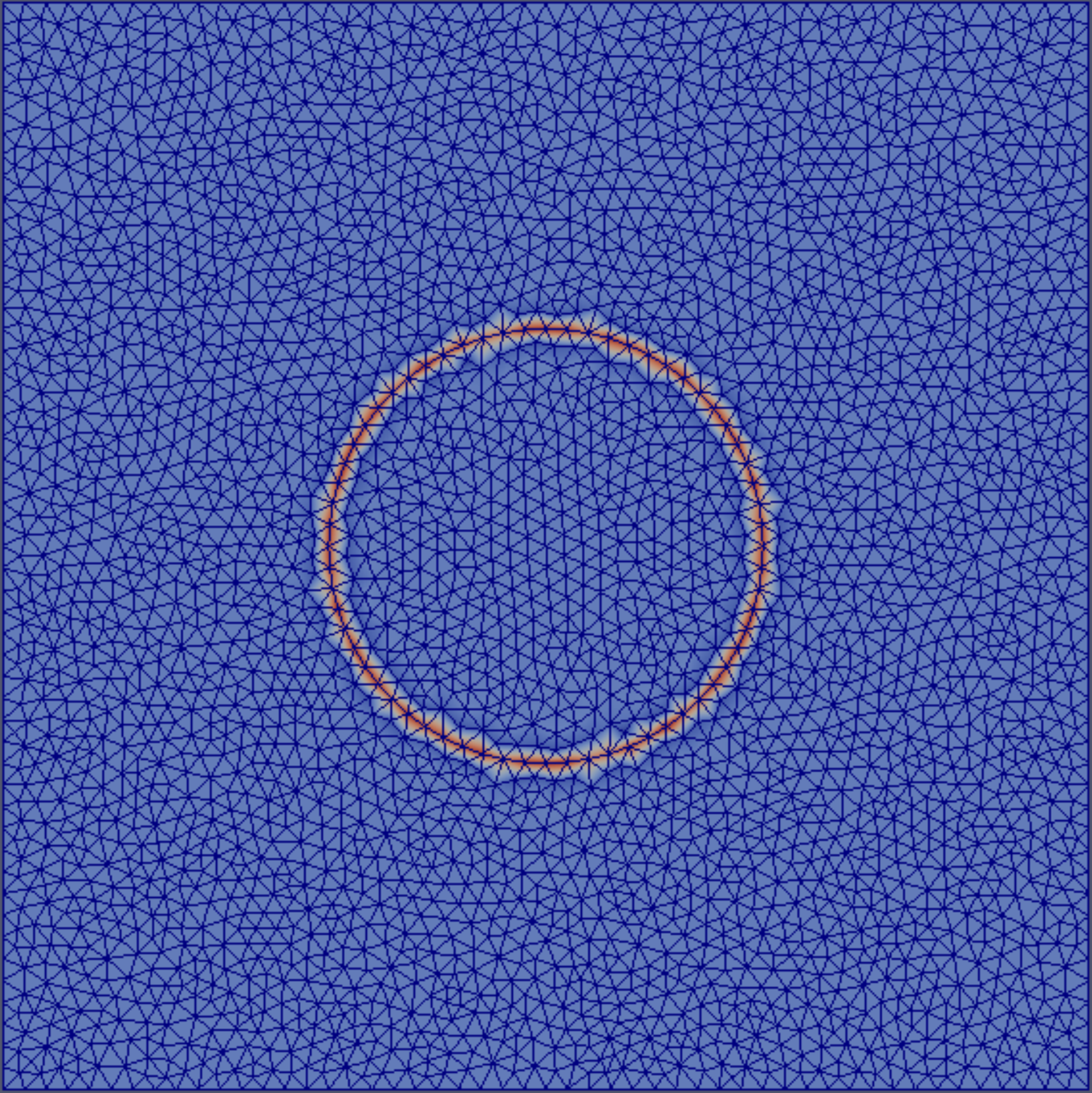}
\caption{A triangulation of the unit square with the curve $g\circ S^1$ highlighted to show its alignment with the mesh skeleton.}
\label{fig:g}
\end{figure}

\defproblem{Weak tri-Helmholtz}{\label{weaktri}
Find $u\in H_0^m(\Omega)^d$ such that:
\begin{equation}
    a(u, v) = \int_{S^1} \tilde f\cdot v\circ g \dth, \qquad  \forall v \in H^3(\Omega)^d,\label{Ham:cont}
\end{equation}
where $\tilde f \in \Fspace$.
}
\begin{remark}[On global regularity]
\label{rmk:reg} Recall that $d=2$ and that
$g\in\mathsf{C}^0(S^1)^d$ forms an embedded curve $G=g\circ S^1$ in a
sufficiently large bounded domain $\Omega$ (away from the boundary
$\partial\Omega$), and let $\oin$, $\oout$ denote the
inside and outside of the curve. Suppose further that $g$ is such that we can
define trace operators $\gamma_G^-: \oin\rightarrow H^{1/2}(G)$ and
$\gamma_G^+: \oout\rightarrow H^{1/2}(G)$. A functional of the form:
\[
H^1(\Omega)\ni w \mapsto \int_{S^1} \gamma_G^-[w]\circ g\dth,
\]
occupies $H^{-1}(\Omega)$, and, \emph{a fortiori}, in
$H^{-3}(\Omega)$. It is of course equivalent to use the trace $\gamma_G^+$ in
the example above.  As such, these cannot be represented as inner products in
$L^2(\Omega)$. The $H^{-1}(\Omega)$ regularity prohibits \emph{global}
higher-order elliptic regularity beyond $H^3(\Omega)$ of the solution $w$ to
\eqref{Ham:cont}. We refer to
\cite{grisvard1992singularities,kondrat1967boundary} for a priori regularity
estimates for general elliptic equations.
\end{remark}

We write the broken analogue $a_h$ of \eqref{eq:ctsbilin}:
\begin{equation}\label{def:a_h}
a_h(u, v) =\sum_{i=1}^d\sum_{K\in \Omegah}\int_K \sum_{j=0}^m b_j \binom{m}{j}\langle D^j u^i,
D^j v^i\rangle \diff x,
\end{equation}
This defines the following \emph{energy norm} as well:
\[
\hnorm{\cdot}^2=a_h(\cdot,\cdot)\,,
\]
and we observe the following norm equivalences based on boundedness of $b$:
\[
\sqrt{a_h(\cdot,\cdot)}\simeq\|\cdot\|_{m, \Omegah}, \qquad \sqrt{a(\cdot,
\cdot)} \simeq \|\cdot\|_{m, \Omega}\,.
\]

\defproblem{Discrete Weak tri-Helmholtz}{\label{weaktri:disc}

Find $u_h\in\femspace$ such that:
\begin{equation}\label{Ham:disc}
 a(u_h, v_h) = \int_{S^1} \tilde f\cdot v_h\circ g \dth,\qquad  \forall v_h \in\femspace
\end{equation}
where $\tilde f \in \Fspace$.
}

\begin{corollary}\label{cor:wp1}
Provided the right-hand sides of the problems  \eqref{Ham:cont} and \eqref{Ham:disc} are well-defined, there exist unique solutions
$u\in H^3(\Omega)^d$ and $u_h\in\femspace$ to these
equations, respectively.
\end{corollary}
\begin{proof}
The proof is trivial since the continuous bilinear forms $a$ and $a_h$ control all $j^\text{th}$ order terms for $j=0,\ldots, 3$ as is
therefore naturally coercive on their respective spaces.
\end{proof}

We remark that a Poincar\'e-type lemma holds for functions in $\tilde
V_h$:
\begin{lemma}[Poincar\'e {\cite[Lemma 4.3]{wu2019nonconforming}}]
It holds that:
\begin{align}\label{lemma43}
\|v_h\|_{m, \Omegah} \lesssim \|v_h\|_{0,\Omegah} + |v_h|_{m, \Omegah}, \quad
\forall v_h \in \tilde V_h.
\end{align}
\end{lemma}
Using \eqref{lemma43} for each component $i$ of some $v_h \in \femspace$ we
obtain:
\begin{align*}
\| v_h^i \|_{m,\Omegah}^2 \lesssim \ltwonorm{v_h^i}^2 + |v_h^i|_{m,\Omegah}^2,
\end{align*}
so summing over the components $i$ we get:
\begin{align*}
\| v_h \|_{m,\Omegah}^2 \lesssim \langle v_h, v_h\rangle_{0, \Omegah} + \langle v_h, v_h\rangle_{m, \Omegah} \quad\forall v_h\in \femspace\,.
\end{align*}
and as such the results and analysis we present here are equally valid had we selected e.g. $B=\idop-\Delta^m$.\\

We introduce some additional notation:
\begin{itemize}
\item $\Omegah^+$ (resp. $\Omegah^-$) denotes the collection of
$K\in\Omegah$ lying on the outside (resp. inside) of the curve $g \circ S^1$.
\item We let $\mathcal{K}^+$ (resp. $\mathcal{K}^-$) denote the collection of
$K\in\Omegah^+$ (resp. $K\in\Omegah^-$) such that $K \cap g\circ S^1
\neq\emptyset$.
\item For $e\in\mathcal{E}_h$, $\mathcal{K}_e$ denotes the collection of
$K\in\Omegah$ such that $K \cap e\neq\emptyset$.
\item For any piecewise affine function $g\in\mathsf{C}^0(S^1)^d$ we define the finite index set $\Jindex$ such that for each $i\in\Jindex$, $g$ can be written as a linear function on some connected $S^1_i$ such that $S^1_i\subset S^1$ and $\cup_{i\in\Jindex} S_i^1 = S^1$ and for any $i, j\in\Jindex$ we have $S^1_i\cap S^1_j = \emptyset$ unless $i=j$.
\item 
For an edge $e$ on $\mathring{\mathcal{E}}_h$ between two triangles $K^-$ and $K^+$ we let the jump $\jump{\cdot}$ and average $\avg{\cdot}$ of a function across $e$ be defined by:
\begin{align*}
\jump{v} & =
\begin{cases}
v^+ \eta^+ +  v^- \eta^- \qquad \quad \, \;\textnormal{if $v$ is scalar,}\\
v^+ \cdot\eta^+ +  v^- \cdot\eta^- \qquad \textnormal{if $v$ is vector-valued,}
\end{cases}\\
\avg{v} & = \half (v^+ + v^-),
\end{align*}
where $v^+$ and $v^-$ are the restrictions of $v$ to $K^-$ and $K^+$, respectively.
On an edge $e\in\partial\Omega_h$, $\jump{v} = \eta v$ and $\avg{v} = v$ where $\eta$ is the
outward normal.
\end{itemize}

We recall the Strang lemma \cite{strang1972variational} lemma (see also \cite[Section 2.2.3]{boffi2013mixed}). In the present context we can state it as follows:
\begin{lemma}[Strang]\label{strang}
Let $u$ and $u_h$ be the solutions of \eqref{Ham:cont} and \eqref{Ham:disc}, respectively. Then,
\begin{equation}\label{eq:strang}
\hnorm{u_h- u} \lesssim \inf_{\mathsf{v}_h \in\femspace} \|u-\vsf_h\|_{m, \Omegah}
+ \sup_{\vsf_h\in \femspace} \frac{|\int_{S^1} \tilde f
\cdot \vsf_h \circ g\dth
- a_h(u, \vsf_h)|}{\normmdh{\vsf_h}}\,,
\end{equation}
where the first term on the right-hand side is called the \emph{approximation}
term and the second a \emph{consistency} term.\\
\end{lemma}
\begin{proof}
Let $\vsf_h, \wsf\in\femspace$:
\begin{align*}
a_h(\uth - \wsf, \vsf_h) & = a_h(\uth - u, \vsf_h)+a_h(u-\wsf,\vsf_h)\\
& = \int_{S^1} \tilde f \cdot \vsf_h \circ g\dth - a_h(u, \vsf_h)+a_h(u- \wsf,\vsf_h)\,.
\end{align*}
Dividing by $\hnorm{\vsf_h}$ and using coercivity and continuity of $a_h$ we can
therefore claim:
\begin{align*}
\hnorm{u_h-\wsf} \lesssim \sup_{\vsf_h\in \femspace}\frac{
\int_{S^1} \tilde f \cdot \vsf_h \circ g\dth
- a_h(u, \vsf_h)|}{\hnorm{\vsf_h}}
+  \|u-\wsf\|_{m, \Omegah}\,,
\end{align*}
which together with the triangle inequality yields the result.
\end{proof}
We say that a discretisation for which the terms on the right-hand side of
\eqref{eq:strang} vanish as $h\rightarrow 0$ is an \emph{approximate} and
\emph{consistent} discretisation. Approximation is guaranteed via classic
element-wise interpolation estimates - we briefly state the main results from
\cite{wu2019nonconforming} adapted to the current setting.

\begin{definition}[Canonical interpolation operator]
For a given $K\in\Omegah$ let us denote by $\Pi_K:
H^{m}(K)\rightarrow \tilde{P}_K^{(3,2)}$, the canonical interpolation operator
where $\tilde{P}_K^{(3,2)} \subset \mathcal{P}^7(K)$ is a seventh 
order polynomial space defined in \cite[Equation 4.1]{wu2019nonconforming}.
Section 4 of \cite{wu2019nonconforming} describes this space in great detail. A
technical result \cite[Lemma 4.1]{wu2019nonconforming} determines the local
degrees of freedom that show unisolvency of $\tilde{P}_K^{(3,2)}$ i.e. those
that uniquely determine any $v\in \tilde{P}_K^{(3,2)}$. This unisolvency
property allows us to sketch the definition of $\Pi_K$:
\[
\Pi_K v = \sum_{j=1}^\mathfrak{m} l_{K,j} n_{K,j}(v),\qquad \forall v\in H^{m}(K),
\]
where $\{ l_{K,j}\}_{j=1}^\mathfrak{m}$ is a local basis and $\{ n_{K,j} \}_{j=1}^\mathfrak{m}$ (for
some positive integer $\mathfrak{m}$) is a nodal basis satisfying $n_{K,i}(l_{K,k}) =
\delta_{ik}$, where $\delta_{ik}$ is the Kronecker delta. See section 2.3 of the
aforementioned reference for an example of a constructing of such a local
interpolant.
\end{definition}

By standard interpolation theory (\cite{brennerscott}, \cite[Lemma
2.4]{wu2019nonconforming}) we have:
\begin{lemma}\label{lemma:localK}
For $s\in [0,1]$ and $k$ such that $0\leq k \leq m$ and $K\in\Omegah$:
\[
|v-\Pi_K|_{k, K} \lesssim h_K^{m+s-k} |v|_{m+s,K},\quad \forall v\in H^{m+s}(K).
\]
\end{lemma}

The \emph{global} operator $\pi_h$ on $H^m(\Omega)^d$ into $\femspace$ is
defined by $(\pi_h v)|_K = \Pi_h|_K (v|_K)$ for $K \in \Omegah$ and satisfies 
the following standard result:
\begin{theorem}\label{thm25}
For $v\in H^{m+s}(\Omega)$, $s \geq 0$ we have:
\begin{align*}
& \|v - \pi_hv\|_{m, \Omegah} \leq h^s |v|_{m+s,\Omega},\\
& \lim_{h\rightarrow 0} \|v - \pi_h v\|_{m, \Omegah} = 0.
\end{align*}
\end{theorem}
\begin{proof}
The first property is obvious from lemma \ref{lemma:localK}. For the second, see
\cite[Theorem 2.5]{wu2019nonconforming}.
\end{proof}

This states that globally $m+s$ weakly differentiable functions can be
well-approximated in the norm $\|\cdot\|_{m, \Omegah}$.  However, by remark
\ref{rmk:reg}, global higher-order regularity of solutions to \eqref{Ham:cont} 
is not possible. Local regularity may, however, be
recovered by the following lemma.

\begin{lemma}[Hypoellipticity]\label{lemma:local_regularity}
Let $u$ be the solution to \eqref{Ham:cont} and define by $\oin$, $\oout$, (resp. $\oinh$, $\oouth$) the connected open sets bounded by $g\circ S^1$ and $\partial\Omega$ (resp. $\partial\Omegah$ and $g\circ S^1$) on the inside and outside of the curve. Then the restriction of $u$ to the sets $\oin$, $\oout$ is $\mathsf{C}^\infty$ in each component.\\
\end{lemma}
\begin{proof}
We recall that a differential operator $\mathfrak D$ is hypoelliptic
\cite{brezis2010functional} if for any open set $\omega$,
$\mathsf{C}^\infty(\omega)\ni \mathfrak Dz \Rightarrow z \in
\mathsf{C}^\infty(\omega)$ \cite[Theorem
6.33]{folland1995introduction}\footnote{For completeness we highlight that
although this result is stated for distributions it is also valid for the kind of
Hilbert-Sobolev spaces we consider here and we refer to the comment on
this extension near the end of part C of Chapter 6 of
\cite{folland1995introduction}.}. Moreover, any constant coefficient elliptic
operator is in fact hypoelliptic \cite[Corollary 6.34]{folland1995introduction}.
We observe that the operator $B$ associated to the bilinear form $a(\cdot,
\cdot)$ is hypoelliptic on open sets, and so in fact, for any open set
$\omega\subset\Omega$ such that $\bar\omega\subset\Omega$ is compact and such
that $\omega\cap g\circ S^1 = \emptyset$, $Bu|_\omega \in
\mathsf{C}^\infty(\omega)$. In the case of the curve
$g \circ S^1$ it is easy to see that for any $K\in\Omegah$ we can solve an
elliptic problem with smooth data on the interior and recover the local
estimate.  The key property here is convexity of $K$, see e.g.  \cite[Chapter
4]{grisvard2011elliptic}.
\end{proof}
From lemma \ref{lemma:local_regularity} we can state the following nonconforming estimate using theorem \ref{thm25} applied
in an element-wise manner.
\begin{corollary}\label{cor:approx}
Let $v \in H^{m}(\Omega)$ be such that $v|_K \in H^{m+s}(K)^d$, $\forall K
\in\Omegah$, with $m=d+1$, $s \geq 0$. Then:
\begin{align*}
& \|v - \Pi_h v\|_{m, \Omegah} \leq h^s |v|_{m+s,\Omegah},\\
& \lim_{h\rightarrow 0} \|v - \Pi_h v\|_{m, \Omegah} = 0.
\end{align*}
\end{corollary}

Corollary \ref{cor:approx} establishes an important local approximation property
for the discretisation \eqref{Ham:disc}. We recall that consistency of the
discretisation means estimating to what extent, as a function of $h$, the continuous solution \eqref{Ham:cont} satisfies the equation \eqref{Ham:disc}. This property can be verified by e.g. the generalised patch test (GPT) of Stummel
\cite{stummel1979generalized} by means of the compactness condition given in definition \ref{def:consist} via the result in \cite{wang2013minimal}.\\

The proof of consistency of $\tilde{V}_h$ was omitted in
\cite{wu2019nonconforming} so for completeness we apply an argument similar to
\cite[Theorem 3.2]{wang2013minimal} using the weak continuity lemma
\ref{lemma:wkcts} to show this.  This property states that \emph{integrals} of
derivatives of functions in $\femspace$ are smooth across interior edges of a
triangulation, rather than smoothness in a pointwise sense. 
\begin{remark}[Bibliographic note]
Weak continuity is indeed sufficient (but not necessary) for the
so-called ``\textbf{SPT}'' condition mentioned in Theorem 3.2 of \cite{wang2013minimal}. The
following theorems are not novel and are simply included to benefit any
unfamiliar readers with the literature.
\end{remark}

\begin{theorem}\label{thm:consistency}
$\femspace$ is a consistent approximation of $H^3(\Omega)^d$.
\end{theorem}
\begin{proof}
We easily verify coercivity and continuity of $a$ and $a_h$ by
corollary \ref{cor:wp1} required by the GPT \cite[Section 1.4, Equation
7]{stummel1979generalized}. Using the form of the test as stated in
\cite[Equation 4.4]{wang2001necessity}, $\tilde{V}_h$ passes the GPT if and only
if the following condition stated is satisfied:
\begin{equation}\label{eq:gpt-cond}
\lim_{h\rightarrow 0} \sup_{\substack{v_h\in\tilde{V}_h \\ \hnorm{v_h}\leq 1}} |
T_{\bfalpha, i} (\psi, v_h) | = 0,\quad |\mathbf{\bfalpha}|<m,\quad 1\leq i\leq d,
\quad \forall \psi \in \mathsf{C}^\infty(\bar{\Omega}),
\end{equation}
where:
\[
T_{\bfalpha, i} (\psi, v_h) = \sum_{K\in \Omegah} \int_K \big[\psi \frac{\partial}{\partial
x_i} \partial^\bfalpha v_h + \frac{\partial\psi}{\partial x_i} \partial^\bfalpha v_h 
\big]\diff x\,.
\]
Equivalently, where $\eta_K$ is the outward normal of $K$:
\begin{equation}\label{eq:test2}
T_{\bfalpha, i} (\psi, v_h) = \sum_{K\in \Omegah} \int_{\partial K} \psi
\partial^\bfalpha v_h \eta_K^i \diff s\,.
\end{equation}
For an edge $e\in\meshskelint$ between $K, K'\in\Omegah$ and let $\eta_K$ (resp.
$\eta_{K'}$) denote the outward normal to $K$ (resp. $K'$). We can then write
\eqref{eq:test2} as:
\begin{align*}
T_{\bfalpha, i} (\psi, v_h) & = \sum_{e\in\meshskelint} \int_e \psi
\partial^\bfalpha v_h|_{K} \eta_{K}^i + \psi \partial^\bfalpha v_h|_{K'}
\eta_{{K'}}^i \diff s + \sum_{ e\in\partial\Omegah} \int_e \psi
\partial^\bfalpha v_h \eta_{\partial\Omegah}^i \diff s\\
& = \sum_{e\in\meshskelint} \int_e \psi\big[\partial^\bfalpha v_h|_{K}  -
 \partial^\bfalpha v_h|_{K'} \big]\eta_{{K'}}^i \diff s
 + \sum_{e\in\partial\Omegah} \int_e \psi \partial^\bfalpha v_h
 \eta_{\partial\Omegah}^i \diff s\\
& \lesssim \sum_{e\in\meshskelint}  \|\psi\|_{L^\infty(e)} \int_e \big[\partial^\bfalpha v_h|_{K} -
 \partial^\bfalpha v_h|_{K'} \big] \diff s\\
& \qquad + \sum_{e\in\partial\Omegah} \|\psi\|_{L^\infty(e)} \int_e \partial^\bfalpha v_h \diff s\\
& = 0.
\end{align*}
Here we have used lemma \ref{lemma:wkcts} stating that facet
integral moments of first and second order match between cells, so the
consistency condition is satisfied so $\tilde{V}_h$ passes the GPT, which
implies the same result for $\femspace$ which is therefore a consistent
approximation of $H^3(\Omega)^d$.
\end{proof}

We use the consistent approximation properties of $\femspace$ to show a convergence result for the discretisation in problem \ref{weaktri:disc}. First we need to prove a technical lemma:
\begin{lemma}\label{lemma:technical}
Let $\vsf \in H^3(\Omega)^d$ and the sequence $\{\vsf_h\}_{h>0}$ in $\femspace$ be
such that $\lim_{h\rightarrow 0}\hnorm{\vsf-\vsf_h} = 0$. Then, for any
$g\in\mathsf{C}^0(S^1)^d$ such that $g \circ S^1$ is comprised of edges $e\in\meshskelint$ of a shape-regular, quasi-uniform mesh $\Omegah$, $(\vsf_{h_k}-\vsf)\circ g$ converges a.e. pointwise to zero on $S^1$ (with respect to the $1$-dimensional Lebesgue measure).
\end{lemma}

\begin{proof}
We have:
\begin{align*}
\textnormal{ess}\,\sup_{\theta\in S^1}|(\vsf_{h_k}-\vsf)\circ g|(\theta) & = 
\textnormal{ess}\,\sup_{x \in g \circ S^1}|\vsf_{h_k}-\vsf|(x)\\
& = \max_{i\in \mathcal{J}} \textnormal{ess}\,\sup_{x \in g \circ S_i^1}|\vsf_{h_k}-\vsf|(x).
\end{align*}
Since $\vsf_{h_k}-\vsf$ is continuous everywhere, $\vsf_{h_k}-\vsf =
\avg{\vsf_{h_k}-\vsf}$ on edges of the interior mesh skeleton $\meshskelint$. 
Therefore, for any two elements $K_i^+$ and $K_i^-$ joining an edge $g \circ
S_i^1$ we can say:
\begin{align*}
\max_{i\in \mathcal{J}} \textnormal{ess}\,\sup_{x \in g\circ S_i^1}|\vsf_{h_k}-\vsf|(x)
& = \max_{i\in \mathcal{J}} \textnormal{ess}\,\sup_{x \in g\circ
S_i^1}|\avg{\vsf_{h_k}-\vsf}|(x)\\
& \leq \max_{i\in\mathcal{J}}\Big[ \textnormal{ess}\,\sup_{x \in g\circ
S_i^1}|\vsf_{h_k}|_{K_i^+}-\vsf|(x)\\
& \qquad\;\;\, + \textnormal{ess}\,\sup_{x \in g\circ
S_i^1}|\vsf_{h_k}|_{K_i^-}-\vsf|(x)\Big].
\end{align*}
By the Sobolev embedding theorem in one dimension we have for any bounded open
set $U$: $f\in H^1(U)\Rightarrow f\in\mathsf{C}^0(U)$, so:
\begin{align*}
\textnormal{ess}\,\sup_{\theta\in S^1}|(\vsf_{h_k}|_{K_i^+}-\vsf)\circ g|(\theta) &\leq 
\max_{i\in \mathcal{J}} \|\vsf_{h_k}|_{K_i^+}-\vsf\|_{1,g\circ S_i^1}.
\end{align*}
By the trace theorem we have for any $i\in \Jindex$:
\begin{align*}
\|\vsf_{h_k}|_{K_i^+}-\vsf\|_{1,g \circ S_i^1}\lesssim
\|\vsf_{h_k}-\vsf\|_{2,\Omegah}.
\end{align*}
Since $\|\vsf_{h_k}-\vsf\|_{2,\Omegah}$ converges by assumption we
conclude the proof by applying the same argument to the restriction of
$\vsf_{h_k}$ to $K_i^-$.\\

Another way to see this result is to recall that $\femspace|_K\subset L^{\infty}(K)$, $K\in\Omegah$. Since $K$ is convex, \cite[Theorem 5]{stein1970singular} tells us that an extension result exists and we can therefore ascribe meaning to the values of a function in $\femspace|_K$ on $\partial K$.
\end{proof}

\begin{corollary}\label{cor:uconvergence}
Let $u$ and $u_h$ solve \eqref{Ham:cont} and \eqref{Ham:disc}. Then we have:
\[
\lim_{h\rightarrow 0}\hnorm{u-u_h} = 0.
\]
\end{corollary}
\begin{proof}
Essential to our proof is the fact that the curve $\Jindex$ is independent of the mesh resolution. The reader should think about this corollary as a convergence result if we refine the mesh \emph{around} the embedded curve $g$.\\

We know that $\hnorm{u_h}\lesssim \ltwonorms{\tilde f}$ since $\femspace$ is
Lipschitz-conforming cf. theorem \ref{thm:H3Lip}, so $\{u_h\}_{h>0}$ is a
bounded sequence. Using the definition of consistency in
\cite{stummel1979generalized,wang2013minimal} we say that $\femspace$ is a
consistent approximation of $H^3(\Omega)^d$ by theorem \ref{thm:consistency} and
\cite[Section 3.1]{wang2013minimal} so we can extract a subsequence
$\{u_{h_k}\}_{k\geq 0}$ of $\{u_h\}_{h>0}$ and a function $\mathsf{w}' \in
H^3(\Omega)^d$ such that $\partial^\bfalpha u_{h_k}\rightarrow \partial^\bfalpha
\mathsf{w}'$ weakly in $L^2(\Omega)$ for all $|\bfalpha| \leq m$. We know by
corollary \ref{cor:approx} that given any $\vsf \in H^3(\Omega)^d$ there exists
a sequence $\{\vsf_h\}_{h>0}$ in $\femspace$ such that $\lim_{h\rightarrow 0}\hnorm{\vsf-\vsf_h}=
0$. Then we can say:
\begin{align*}
|a_{h_k}(u_{h_k}, \vsf_{h_k}) - a(\mathsf{w}',
\vsf)| & \leq |a_{h_k}(u_{h_k}, \vsf_{h_k}-\vsf)|
+ |a_{h_k}(u_{h_k}-\mathsf{w}', \vsf)|\\
&\leq \hnorm{u_{h_k}}\hnorm{\vsf_{h_k}-\vsf} +  |a_{h_k}(u_{h_k}-\mathsf{w}', \vsf)|
\end{align*}
which vanishes as $k\rightarrow\infty$. Next we show:
\begin{equation}\label{conv_sourceterms}
\int_{S^1} \tilde f \cdot \vsf_{h_k} \circ g \dth \rightarrow \int_{S^1}
\tilde f \cdot\vsf \circ g \dth
\quad \text{ as }\quad h_k\rightarrow 0.
\end{equation}
By the dominated convergence theorem this requires at least
pointwise convergence of $(\vsf_{h_k}-\vsf)\circ g$ and pointwise boundedness
by some integrable function. The latter is trivial since $\vsf_{h_k}$ and $\vsf$
are continuous over a bounded domain. The information we have is strong convergence of $\|\vsf_{h_k}-\vsf\|_{m,\Omegah}$ by the weak approximation property, so by
lemma \ref{lemma:technical} \eqref{conv_sourceterms} is verified. In light of this we must have $a(\mathsf{w}', \vsf) = \int_{S^1}\tilde f \cdot \vsf \circ g \dth$ and so corollary \ref{cor:wp1} implies $\mathsf{w}'\equiv
u$.\\

We now show strong convergence. Let $\{\wsf\}_{h>0}$ be a sequence in $\femspace$ such that $\hnorm{u-\wsf} \rightarrow 0$ as $h \rightarrow 0$. Then,
\begin{align*}
\hnorm{u-u_h}^2 & \lesssim\hnorm{u-\wsf}^2 + a_h(\wsf - u_h, \wsf-
u_h)\\
 & \lesssim \hnorm{u-\wsf}^2 + a(u, u) - 2a_h(u, u_h)
+ a_h(u_h, u_h)\\
& \lesssim \hnorm{u-\wsf}^2 + a(u, u) - 2a_h(u, u_h)\\
& + \int_{S^1} \tilde f \cdot u_h \circ g \dth.
\end{align*}
Now using what we just derived above:
\begin{align*}
& a(u, u) - 2a_h(u, u_h)
+ \int_{S^1} \tilde f \cdot u_h \circ g \dth\\
\rightarrow \quad &
a(u, u) - 2a_h(u, u)
+ \int_{S^1} \tilde f \cdot u \circ g \dth = 0
\end{align*}
as $h\rightarrow 0$.
\end{proof}

\begin{remark}[Comparison with $L^2(\Omega)$ source terms]
When studying elliptic equations it is common for the source term to be defined in terms of some $f\in L^2(\Omega)^d$ on the form $\langle f, \vsf_{h_k}\rangle_{0,\Omega}$. Weak
$L^2(\Omega)^d$ convergence therefore trivially completes the argument in this
case since:
\[
\langle f, \vsf_{h_k}\rangle_{0,\Omega} \rightarrow
\langle f, \vsf\rangle_{0,\Omega},
\]
is by definition, weakly convergent. This argument is ill-suited to the setting
above due to the singular nature of the $S^1$ integral as a source term. This is
the reason for the use of lemma \ref{lemma:technical} in corollary
\ref{cor:uconvergence}.
\end{remark}

The key thing to note in the result above is that the implicit conforming shape-regular
family of triangulations
$\{\Omegah\}_{h_j}$ is such that for any $h_j\rightarrow 0$ as
$j\rightarrow\infty$ there is a subset of $\mathring{\mathcal{E}}_{th_j}$
tracing out $g\circ S^1$.\\

We proceed to prove a convergence rate. To this end we introduce the notion
of \emph{conforming relatives} first proposed in \cite{brenner1996two}.

\begin{proposition}[Conforming relatives]
There exists an $H^m$-conforming finite element space $V_h^{c,d}\subset
H_0^m(\Omega)^d$ and an operator $\CR:\femspace \rightarrow V_h^{c,d}$ such
that $\forall v\in\femspace$:
\begin{equation}\label{relatives}
\sum_{j=0}^{m-1} h^{2(j - m)} |v_h^i - \big[\CR v_h \big]^i|_{j,\Omegah}^2 +
|\big[\CR v_h \big]^i|_{m,\Omegah}^2 \lesssim |v_h^i|_{m,\Omegah}^2.
\end{equation}
for each component $v_h^i$ of $v_h$.
\end{proposition}
This result can be seen in e.g.  \cite{hu2017canonical} or \cite[Lemma
3.2]{hu2014new}\footnote{The reference given here shows under which conditions
the conforming relatives operator exists, for arbitrary $m\geq 1$. These are
exactly the weak continuity properties provided by \cite{wu2019nonconforming}.}.
The proof we provide next uses these operators as well as the hypoellipticity
result in lemma \ref{lemma:local_regularity}.

\begin{corollary}\label{cor:strang}
The convergence rate for the expression $\lim_{h\rightarrow
0}\hnorm{u-\uth} = 0$ is almost linear in $h$ via the following bound:
\begin{align*}
\hnorm{u-\uth} & \lesssim h\Big(\ltwonorms{\tilde f} + \sum_{i=1}^2 \ltwonormh{D^3u^i}
+ \ltwonormh{D^5u^i}
+ \ltwonormh{D^4u^i}\\
&+h \big(\sum_{e\in g \circ S^1} \ltwoe{\jump{D^3 u^i}} + 
\ltwoe{\jump{D^5 u^i}}\big)  + 2 \ltwonormh{D^4u^i}\Big)\\
& + \sum_{e\in g \circ S^1}\ltwoe{\jump{D^3 u^i
-P_{\mathcal{K}_e}\big[D^3 u^i\big]}}\\
& +h^s |u|_{m+s,\Omegah},
\end{align*}
for any integer $s\geq 0$, where $P_{\mathcal{K}_e}$ is the projection operator
on $L^2(\mathcal{K}_e)$ defined by:
\begin{align*}
P_{\mathcal{K}_e}v = (\sum_{K\in \mathcal{K}_e}|K|)\inv \sum_{K\in
\mathcal{K}_e}\int_K v \dx,\qquad v\in L^2(\Omega).
\end{align*}
\end{corollary}
\begin{proof}
For arbitrary $\wsf\in\femspace$ we seek to bound the following expression in terms of $h$:
\begin{align*}
a_h(u-u_h, \wsf) & = a_h(u, \wsf) - \int_{S^1}
\tilde f\cdot \wsf\circ g\dth.
\end{align*}
We use the conforming relatives operator to say:
\begin{align*}
a_h(u-u_h, \wsf) & = a_h(u, \wsf) - \int_{S^1} \tilde f\cdot \wsf\circ g\dth\\
 & = a_h(u, \wsf-\CR\wsf) - \int_{S^1} \tilde f\cdot
 (\wsf-\CR\wsf)\circ g\dth,
\end{align*}
since $a_h(u,\CR\wsf) - \int_{S^1}\tilde f\cdot\CR\wsf\circ g\dth=0$.\\

First we estimate $\int_{S^1} \tilde f\cdot(\wsf-\CR\wsf)\circ g \dth$ and show
linear convergence in $h$. The key tools are the Sobolev embedding theorem and
\eqref{relatives}:
\begin{align*}
\int_{S^1}\tilde f\cdot(\wsf-\CR\wsf)\circ g \dth
& \leq \ltwonorms{\tilde f} \ltwonorms{(\wsf-\CR\wsf)\circ g}\\
& \lesssim \ltwonorms{\tilde f} \linftysd{(\wsf-\CR\wsf)\circ g}\\
& = \ltwonorms{\tilde f} \|\wsf-\CR\wsf\|_{\infty,  g\circ S^1}\\
& \lesssim \ltwonorms{\tilde f} \|\wsf-\CR\wsf\|_{1,  g\circ
S^1}\;\textnormal{(embedding)}\\
& \lesssim \ltwonorms{\tilde f} \|\wsf-\CR\wsf\|_{2, \Omegah}
\end{align*}
Now using the bound in \eqref{relatives}:
\begin{equation}\label{eq:ptermc4}
\int_{S^1} \tilde f\cdot(\wsf-\CR\wsf)\circ g \dth\lesssim h\ltwonorms{\tilde f} \hnorm{\wsf}.
\end{equation} 

Second we estimate $a_h(u, \wsf-\CR\wsf)$ by writing the variational
problem for $w$ in a strong form. For the sake of exposition we
recall the standard integration by parts identities on elements $K$ of
$\Omegah$ in our notation:
\begin{align*}
 \int_{K} \langle D u^i, D \wsf^i \rangle \diff x
&= - \int_{K} \langle D^2 u^i, \wsf^i\rangle  \diff x
 +  \intpk \langle D u^i, \eta_{K} \wsf^i\rangle  \diff s,\\
\int_{K} \langle D^2 u^i, D^2 \wsf^i\rangle  \diff x
 &= \int_{K} \langle D^4 u^i, \wsf^i\rangle  \diff x
 +  \intpk \langle \eta_{K}D^2 u^i, D\wsf^i\rangle  \diff s\\
 & -  \intpk  \langle D^3  u^i, \eta_{K} \wsf^i\rangle \diff s,\\
\int_{K} \langle D^3 u^i, D^3 \wsf^i \rangle\diff x
 &= - \int_{K} \langle D^6 u^i, \wsf^i \rangle\diff x
 +  \intpk \langle D^5 u^i,\eta_{K} \wsf^i \rangle\diff s\\
 &-  \intpk \langle \eta_{K} D^4 u^i, D\wsf^i \rangle \diff s
  +  \intpk \langle D^3 u^i, \eta_{K} D^2\wsf^i\rangle  \diff s,
\end{align*}

Recall that $\oinh$ (resp. $\oouth$) denotes the parts of domain $\Omegah$
restricted to the inside (resp. outside) of the curve traced by $g \circ S^1$.
\begin{remark}[Strong form equation]
The strong form version of \eqref{Ham:cont} can be written by using the
identities above:
\begin{align*}
& B u = 0, \quad \textnormal{on}\quad  \oinh \textnormal{ and }
\oouth,\\
& \int_{S^1} \tilde f\cdot v\circ g \dth = \sum_{i=1}^d
\sum_{e\in\gam} \intep  b_2\binom{m}{2}\langle \jump{D^3 u^i}, v^i\rangle  - b_3\binom{m}{3} \langle\jump{D^5 u^i
},v^i\rangle\nonumber \\ &\qquad\qquad\qquad +
 b_3\binom{m}{3}\langle\jump{D^4 u^i}, D v^i\rangle - b_3\binom{m}{3}\langle \jump{D^3
 u^i},D^2 v^i\rangle \diff s,\\
 &\qquad\qquad\qquad  \forall  v\in H^3(\Omega).
\end{align*}
\end{remark}

Indeed, for a test function $v \in H^3(\Omega)^d$ we have:
\begin{align*}
 & \sum_{i=1}^d\int_{\Omegah} \sum_{j=0}^m b_j\binom{m}{j} \langle D^j u^i,
 D^j v^i\rangle  \diff x\\
=& \sum_{i=1}^d \int_{\oinh}  \sum_{j=0}^m b_j\binom{m}{j}\langle D^j u^i, D^j v^i\rangle \diff
x+\int_{\oouth}  \sum_{j=0}^m b_j\binom{m}{j}\langle D^j u^i, D^j v^i\rangle \diff
x.
\end{align*}
So integrating by parts:
\begin{align*}
 & \sum_{i=1}^d\int_{\Omegah} \sum_{j=0}^m b_j\binom{m}{j} \langle D^j u^i,
 D^j v^i\rangle  \diff x\\
=& \sum_{i=1}^d\int_{\oinh} Bu^i\cdot  v^i \diff x
+\sum_{i=1}^d\int_{\oouth} Bu^i\cdot  v^i \diff x\\
& + \int_{\circ S^1}b_2 \binom{m}{2}\langle \jump{D^3 u^i}, v^i\rangle - b_3\binom{m}{3}\langle\jump{D^5 u^i
},v^i\rangle\\
& \qquad\quad  +  b_3\binom{m}{3}\jump{D^4 u^i
D v^i} - b_3\binom{m}{3} \jump{D^3 u^i  D^2 v^i} \diff s\,.
\end{align*}
Since the integral over $\oinh$ and $\oouth$ of $Bu^i\cdot v^i$ vanishes we can relate the singular source term to a jump condition on $g\circ S^1$ via the
following variational equality:
\begin{align}\label{eq:vareq}
\int_{S^1} \tilde f \cdot v\circ  \dth & = \sum_{i=1}^d
\sum_{e\in\gam} \intep b_2 \binom{m}{2}\langle \jump{D^3 u^i}, v^i\rangle  - b_3\binom{m}{3}\langle\jump{D^5 u^i
},v^i\rangle  \nonumber \\ &\qquad\qquad\qquad +  b_3\binom{m}{3}\langle\jump{D^4 u^i}, D v^i\rangle \nonumber \\ &\qquad\qquad\qquad - b_3\binom{m}{3}\langle \jump{D^3
 u^i},D^2 v^i\rangle \diff s,\quad \forall
 v\in H^3(\Omega)\,.
\end{align}
Then:\\
\begin{align*}
a_h(u, \wsf-\CR\wsf) & = \sum_{i=1}^2 \sumks 
\underbrace{\int_{K}B u^i\cdot (\wsf^i - \CRi) \diff x}_{\textnormal{sum
cancels by definition}}\\
&  +  \underbrace{\intpk b_1\binom{m}{1}\langle D u^i, \eta_{K} (\wsf^i - \CRi)\rangle \diff
 s}_{\textnormal{sum cancels because integrand is $H^1$}}\\
& +  \intpk b_2\binom{m}{2}\langle \eta_{K}D^2 u^i, D(\wsf^i - \CRi)\rangle
\diff s\\
&  -  \intpk b_2\binom{m}{2} \langle D^3 u^i, \eta_{K} (\wsf^i - \CRi)\rangle \diff s\\
& -  \intpk b_3\binom{m}{3}\langle D^5 u^i, \eta_{K} (\wsf^i - \CRi)\rangle
\diff s\\
& +  \intpk b_3\binom{m}{3}\langle \eta_{K}D^4 u^i, D(\wsf^i - \CRi)\rangle \diff s\\
 &-  \intpk b_3\binom{m}{3}\langle D^3 u^i, \eta_{K} D^2(\wsf^i - \CRi)\rangle \diff s.
\end{align*}

We now write the boundary terms as jumps on the mesh skeleton. The expressions
$\wsf^i - \CRi$ and $D^2 u^i$ have unique traces everywhere, while the
expressions $D^ju^i$, $j=3,4,5$ only have unique traces on
$\meshskelng$. Lastly, the expressions $D^j(\wsf^i - \CRi)$, $j=1,2$ are not
uniquely defined anywhere on $\meshskelint$. We can therefore write:\\
\begin{subequations}\label{thatlong}
\begin{align}
& a_h(u, \wsf-\CR\wsf) \nonumber \\
& = \sum_{i=1}^2 
\int_{\meshskelint} b_2\binom{m}{2} \langle D^2 u^i, \jump{D(\wsf^i -
\CRi)}\rangle \ds\label{term1}\\
& + \int_{\meshskelng} b_3\binom{m}{3}\langle D^4 u^i, \jump{D(\wsf^i -
\CRi)}\rangle\label{term2}\\
& \quad\quad\qquad - b_3\binom{m}{3} \langle D^3 u^i, \jump{D^2(\wsf^i - \CRi)}
\rangle\ds\label{term3}\\
& + \int_{\circ S^1} - b_2 \binom{m}{2}\langle \jump{D^3 u^i}, \wsf^i - \CRi\rangle\label{term4}\\
& \,\;\qquad\qquad - b_3 \binom{m}{3}\langle \jump{D^5 u^i}, \wsf^i - \CRi\rangle\label{term5}\\
& \,\;\qquad\qquad + b_3\binom{m}{3} \jump{D^4 u^i D(\wsf^i - \CRi)}\label{term6}\\
& \,\;\qquad\qquad- b_3 \binom{m}{3}\jump{D^3 u^i D^2(\wsf^i - \CRi)}\ds,\label{term7}
\end{align}
\end{subequations}

where we recall that $\jump{D^4 u^i D(\wsf^i - \CRi)}$ and
$\jump{ D^3 u^i D^2 (\wsf^i - \CRi)}$ on $e = \Kep\cap\Kem$ are defined as:
\begin{align*}
 \jump{D^4 u^i D (\wsf^i - \CRi)} &= \langle \eta_\Kep D^4 u^i|_\Kep ,  D
(\wsf^i - \CRi)|_\Kep \rangle\\
&  + \langle \eta_{\Kem}D^4 u^i|_{\Kem},  D (\wsf^i - \CRi)|_{\Kem}\rangle,\\
 \jump{ D^3 u^i D^2 (\wsf^i - \CRi)}  &= \langle \eta_\Kep  \cdot D^3
u^i|_\Kep ,  D^2 (\wsf^i - \CRi)|_\Kep \rangle\\
&+ \langle \eta_{\Kem}\cdot D^3 u^i|_{\Kem},  D^2 (\wsf^i - \CRi)|_{\Kem}\rangle.
\end{align*}

Many of the terms in \eqref{thatlong} are estimated using the same technique and
rely on the weak continuity result for $\femspace$ in lemma \ref{lemma:wkcts}.
First we highlight the main inequalities on which we rely. Let $e$ be an edge of
an arbitrary $K\in\Omegah$ with length $|e|$. Further, let $\mathcal{K}_e$
denote the set of triangles in $\Omegah$ having $e$ as an edge. We
recall \cite[Equations 10.3.8, 10.3.9]{brennerscott} and the trace theorem:
\begin{subequations}\label{terms17}
\begin{align}
|e|\inv\ltwoe{f}^2 & \lesssim h_K^{-2}\ltwoK{f}^2
+|f|_{1,K}^2,\quad f\in H^1(K)\label{terms17a}\\
|e|\ltwoe{\jump{w}}^2 & \lesssim h_K^2
\sum_{K\in\mathcal{K}_e}\honeK{w}^2,\quad w|_K\in H^1(K),\label{terms17b}\\
 \ltwoe{f|_K}^2 & \lesssim \honeK{f}^2,\quad\qquad\qquad f|_K\in H^1(K).
\end{align}
\end{subequations}


We now estimate each term in \eqref{thatlong} starting with \eqref{term1}. We
now drop the binomial coefficients to ease the notation.\\

For any constant $c_e$ and using Cauchy-Schwartz:
\begin{align*}
& \sum_{i=1}^2\int_{\meshskelint} b_2 \langle D^2 u^i, \jump{D(\wsf^i -
\CRi)}\rangle \ds\\
= & \sum_{i=1}^2\int_{\meshskelint} b_2 \langle D^2 u^i-c_e, \jump{D(\wsf^i -
\CRi)}\rangle \ds\\
 \leq & b_2 \sum_{i=1}^2\sum_{e\in\meshskelint} \ltwoe{D^2 u^i-c_e}
 \ltwoe{\jump{D(\wsf^i - \CRi)}}.
\end{align*}
Now we multiply by $1=|e|^{-1/2}|e|^{1/2}$:
\begin{align*}
& \sum_{i=1}^2\int_{\meshskelint} b_2 \langle D^2 u^i, \jump{D(\wsf^i -
\CRi)}\rangle \ds\\
=&  b_2 \sum_{i=1}^2\sum_{e\in\meshskelint} |e|^{-1/2}\ltwoe{D^2 u^i-c_e} |e|^{1/2}\ltwoe{\jump{D(\wsf^i - \CRi)}}\\
\leq & b_2 \big(\sum_{i=1}^2 \sum_{e\in\meshskelint} |e|^{-1}\ltwoe{D^2
u^i-c_e}^2\big)^{1/2}\big(|e|\ltwoe{\jump{D(\wsf^i -
\CRi)}}^2\big)^{1/2}.
\end{align*}
Now using \eqref{terms17}:
\begin{align*}
& \sum_{i=1}^2\int_{\meshskelint} b_2 \langle D^2 u^i, \jump{D(\wsf^i -
\CRi)}\rangle \ds\\ \lesssim & b_2 
\big(\sum_{i=1}^2 \sum_{e\in\meshskelint} \min_{K\in\mathcal{K}_e} \big[ h_K^{-2}
\ltwoK{D^2u^i - c_e}^2 + |D^2u^i|_{1,K}^2 \big]\big)^{1/2}\\
&\times\big(h_K^2 |D(\wsf^i - \CRi)|_{1,K}^2\big)^{1/2}.
\end{align*}
Now since $c_e$ was arbitrary we can take it to be the constant in $\mathbb R$
that is the closest approximation of $D^2u^i$. Projection of Sobolev
functions onto constants is well-understood; by e.g. \cite[Proposition
1.135]{ErnGuermond2013} we know that:
\[
\ltwoK{D^2u^i - |K|\inv\int_K D^2u^i\dx} \lesssim h_K |D^2u^i|_{1,K},
\]
with a constant independent of $h_K$. Squaring on both sides cancels the
dependence on $h_K$ in the previous estimate so we get:
\begin{align*}
& \sum_{i=1}^2\int_{\meshskelint} b_2  \min_{K\in\mathcal{K}_e}\langle D^2 u^i, \jump{D(\wsf^i -
\CRi)}\rangle \ds\\ \lesssim & b_2 
\big(\sum_{i=1}^2 \sum_{e\in\meshskelint} \min_{K\in\mathcal{K}_e}|D^2u^i|_{1,K}^2\big)^{1/2}
\big(h_K^2 |D(\wsf^i - \CRi)|_{1,K}^2\big)^{1/2}\\
\lesssim & b_2 h \ltwonormh{D^3u^i} \hnorm{\wsf^i - \CRi}\\
\lesssim & b_2 h \ltwonormh{D^3u^i} \hnorm{\wsf},
\end{align*}
using \eqref{relatives} in the last step.\\

Next we estimate \eqref{term2}. It is clear that the only difference from
\eqref{term1} here is the operator $D^4$ instead of $D^2$ applied to $u^i$,
so by the same steps as previously:
\begin{align*}
& \sum_{i=1}^2 \int_{e\in\meshskelng} b_3\langle D^4 u^i, \jump{D(\wsf^i -
\CRi)}\rangle \ds\\
&\lesssim b_3 h \ltwonormh{D^5u^i} \hnorm{\wsf}.
\end{align*}

We now treat \eqref{term3}. For an arbitrary constant $c_e$:
\begin{align*}
& \sum_{i=1}^2 \int_{\meshskelng} b_3\langle D^3 u^i, \jump{D^2(\wsf^i
- \CRi)} \rangle\ds\\
& = \sum_{i=1}^2 \int_{\meshskelng} b_3\langle D^3 u^i - c_e, \jump{D^2(\wsf^i - \CRi)} \rangle\ds\\
& \leq b_3 \big(\sum_{i=1}^2 \sum_{e\in\meshskelng} \ltwoe{D^3 u^i -
c_e}^2\big)^{1/2}\big(\ltwoe{\jump{D^2(\wsf^i - \CRi)}}^2\big)^{1/2}.
\end{align*}
Now using \eqref{terms17} again and Cauchy-Scwartz:
\begin{align*}
& \sum_{i=1}^2 \int_{\meshskelng} b_3\langle D^3 u^i, \jump{D^2(\wsf^i
- \CRi)} \rangle\ds\\
& \leq b_3 \big(\sum_{i=1}^2 \sum_{e\in\meshskelng}
 \min_{K\in\mathcal{K}_e} h_K^{-2}  \ltwoK{D^3 u^i -c_e}^2
 + |D^3u^i|_{1,K}^2 \big)^{1/2}\\
 & \quad \times \big(h_K^2 |D^2(\wsf^i - \CRi)|_{1,K}^2
 \big)^{1/2}\\
& \lesssim b_3 \big(\sum_{i=1}^2 \sum_{e\in\meshskelng} |D^3u^i|_{1,K}^2
\big)^{1/2}\big(h_K^2 |D^2(\wsf^i - \CRi)|_{1,K}^2\big)^{1/2}.
\end{align*}
Using the trace theorem we get:
\begin{align*}
& |\sum_{i=1}^2 \int_{\meshskelng} b_3\langle D^3 u^i, \jump{D^2(\wsf^i
- \CRi)} \rangle\ds|\\
& \lesssim b_3 h \big(\sum_{i=1}^2 \sumks |D^3u^i|_{1,K}^2 \big)^{1/2}
\hnorm{\wsf^i - \CRi}\\
& = b_3 h \ltwonormh{D^4u^i} \hnorm{\wsf^i - \CRi}\\
& \lesssim b_3 h \ltwonormh{D^4u^i} \hnorm{\wsf}.
\end{align*}
Estimating \eqref{term4}, \eqref{term5} and \eqref{term6} relies on the trace
theorem and a property of the conforming relatives operator in
\eqref{relatives}. For \eqref{term4} we have:
\begin{align*}
& \sum_{i=1}^2 \int_{g \circ S^1} - b_2 \langle \jump{D^3 u^i}, \wsf^i -
\CRi\rangle\ds\\ 
\leq & b_2 \sum_{i=1}^2 \sum_{e\in g\circ S^1} 
\ltwoe{\jump{D^3 u^i}}\ltwoe{\wsf^i - \CRi}.
\end{align*}
By the trace theorem we have $\ltwoe{\wsf^i - \CRi}\lesssim
\sum_{K\in\mathcal{K}_e} \honeK{\wsf^i - \CRi}$. Using this and
\eqref{relatives} in the above estimate gives us:
\begin{align*}
& |\sum_{i=1}^2 \int_{g\circ S^1} - b_2 \langle \jump{D^3 u^i}, \wsf^i -
\CRi\rangle\ds |\\ 
\lesssim & b_2 \sum_{i=1}^2 \sum_{e\in g\circ S^1} 
\ltwoe{\jump{D^3 u^i}} \|\wsf^i - \CRi\|_{1, \Omegah}\\
\lesssim & b_2 h^2 \sum_{i=1}^2 \sum_{e\in g\circ S^1} 
\ltwoe{\jump{D^3 u^i}}\hnorm{\wsf}.
\end{align*}
By a similar argument for \eqref{term5}:
\begin{align*}
& |\sum_{i=1}^2 \int_{g\circ S^1} - b_3 \langle \jump{D^5 u^i}, \wsf^i -
\CRi\rangle\ds|\\ 
\lesssim & b_3 \sum_{i=1}^2 \sum_{e\in g\circ S^1} 
\ltwoe{\jump{D^5 u^i}} \|\wsf^i - \CRi\|_{1, \Omegah}\\
\lesssim & b_2 h^2 \sum_{i=1}^2 \sum_{e\in g\circ S^1} 
\ltwoe{\jump{D^5 u^i}}\hnorm{\wsf}.
\end{align*}
Term \eqref{term6} is estimated similarly:
\begin{align*}
& \sum_{i=1}^2\int_{g \circ S^1} b_3 \jump{D^4 u^i D(\wsf^i - \CRi)}\ds\\
= & b_3 \sum_{i=1}^2\int_{g \circ S^1}
\langle \eta_\Kep D^4 u^i|_\Kep ,  D(\wsf^i - \CRi)|_\Kep \rangle\\
& \qquad\quad\quad\;\;+ \langle \eta_\Kem D^4 u^i|_\Kem ,  D (\wsf^i - \CRi)|_\Kem \rangle
\ds.
\end{align*}
Using Cauchy-Schwartz twice, the trace theorem and \eqref{relatives} we proceed
as previously:
\begin{align*}
&  \sum_{i=1}^2\int_{g \circ S^1} b_3 \jump{D^4 u^i D(\wsf^i -
\CRi)}\ds\\
\leq & b_3 \sum_{i=1}^2\sum_{e\in g \circ S^1}
\ltwoe{\eta_\Kep D^4 u^i|_\Kep }\ltwoe{D(\wsf^i - \CRi)|_\Kep }\\
& \qquad\quad\quad\;\; +\ltwoe{\eta_\Kem D^4 u^i|_\Kem }\ltwoe{D (\wsf^i - \CRi)|_\Kem }\\
\lesssim & b_3\sum_{i=1}^2 \sum_{e\in g \circ S^1}
\big(\ltwoe{\eta_\Kep D^4 u^i|_\Kep }
+\ltwoe{\eta_\Kem D^4 u^i|_\Kem}\big)
\|D(\wsf^i - \CRi)\|_{1,\Omegah}\\
\lesssim & b_3 h \Big( \sum_{i=1}^2\sum_{e\in g \circ S^1}
\big(\ltwoe{\eta_\Kep D^4 u^i|_\Kep }
+\ltwoe{\eta_\Kem D^4 u^i|_\Kem}\big)\Big)
\|\wsf\|_{m,\Omegah}\\
\leq & b_3 h \ltwonormh{D^4 u} \hnorm{\wsf}.
\end{align*} 
Finally we estimate \eqref{term7}:
\begin{align*}
& |\sum_{i=1}^2\int_{g \circ S^1} - b_3 \jump{D^3 u^i D^2(\wsf^i -
\CRi)}\ds |\\
= & |-b_3 \sum_{i=1}^2\int_{g \circ S^1}\avg{D^3 u^i}\cdot\jump{D^2(\wsf^i - \CRi)} + \jump{D^3 u^i}\avg{D^2(\wsf^i - \CRi)}\ds|\\
= & T_1 + T_2.
\end{align*}
We estimate $T_1$ using the same method as for the term \eqref{term3}. The only
difference is the presence of the averaging operator, but we can write $\avg{D^3
u^i} = \half D^3u^i|_\Kem + \half D^3u^i|_\Kep$, and since the
restrictions of $D^3u^i$ to $\Kep$ and $\Kem$ are smooth there are no
changes to the procedure for the term \eqref{term3}.
\begin{align*}
T_1 \lesssim b_3 h \ltwonormh{D^4u} \hnorm{\wsf}.
\end{align*}
Obtaining a convergence rate for $T_2$ is more difficult owing to the presence
of the $\avg{D^2(\wsf^i - \CRi)}$ term. It is easy to see why, since for an
arbitrary edge $e\in\meshskelint$, $\ltwoe{D^2(\wsf^i -\CRi)}\lesssim
\honeK{D^2(\wsf^i -\CRi)}$ for some $K\in\Omegah$ with $e$ as an edge. Then by
\eqref{relatives} we know:
\begin{align*}
(\sum_{i=1}^2\sumks \honeK{D^2(\wsf^i -\CRi)}^2)^{1/2}\lesssim\hnorm{\wsf},
\end{align*}
but we do not recover a factor of $h$ in this bound. We can however bound
\eqref{term7} as follows using an averaging operator over a neighbourhood near
an edge (see also \cite[Equation 3.12]{wu2019nonconforming} for an example). We
define $P_{\mathcal{K}_e}$ as the projection operator on $L^2(\mathcal{K}_e)$
defined by:
\begin{align*}
P_{\mathcal{K}_e}v = (\sum_{K\in \mathcal{K}_e}|K|)\inv \sum_{K\in
\mathcal{K}_e}\int_K v \dx,\qquad v\in L^2(\Omega).
\end{align*}
Then,
\begin{align*}
T_2 & = |b_3 \sum_{i=1}^2\int_{g \circ S^1} \jump{D^3 u^i}\avg{D^2(\wsf^i - \CRi)}\ds|\\
& = |b_3 \sum_{i=1}^2\int_{g\circ S^1} \jump{D^3 u^i
-P_{\mathcal{K}_e}\big[D^3 u^i\big]}\avg{D^2(\wsf^i - \CRi)}\ds|\\
& \leq b_3 \sum_{i=1}^2\sum_{e\in g\circ S^1}\ltwoe{\jump{D^3 u^i
-P_{\mathcal{K}_e}\big[D^3 u^i\big]}}\ltwoe{\avg{D^2(\wsf^i - \CRi)}}.
\end{align*}
Now using the trace theorem and \eqref{relatives}:
\begin{align*}
T_2 & \lesssim b_3 \sum_{i=1}^2\sum_{e\in g\circ S^1}\ltwoe{\jump{D^3 u^i
-P_{\mathcal{K}_e}\big[D^3 u^i\big]}} \hnorm{\wsf^i - \CRi}\\
& \lesssim b_3 \sum_{i=1}^2\sum_{e\in g\circ S^1}\ltwoe{\jump{D^3 u^i
-P_{\mathcal{K}_e}\big[D^3 u^i\big]}} \hnorm{\wsf}.
\end{align*}
Collecting the bounds we have for \eqref{thatlong} so far and dividing by
$\hnorm{\wsf}$ yields:
\begin{align*}
\frac{a_h(u - u_h, \wsf)}{\hnorm{\wsf}} \lesssim & h\Big(\sum_{i=1}^2
  \ltwonormh{D^3u^i}
+ \ltwonormh{D^5u^i}
+ \ltwonormh{D^4u^i}\\
&+h \sum_{e\in g\circ S^1}\ltwoe{\jump{D^3 u^i}}\\
& + h
\sum_{e\in g\circ S^1}\ltwoe{\jump{D^5 u^i}}  + 2 \ltwonormh{D^4u^i}\Big)\\
& + \sum_{e\in g\circ S^1}\ltwoe{\jump{D^3 u^i
-P_{\mathcal{K}_e}\big[D^3 u^i\big]}}
\end{align*}

Using Strang's lemma \ref{strang} and corollary \ref{cor:approx} we can
therefore say:
\begin{align*}
\hnorm{u-u_h}& \lesssim h\Big(\ltwonorms{\tilde f}+\sum_{i=1}^2 \ltwonormh{D^3u^i}
+ \ltwonormh{D^5u^i}
+ \ltwonormh{D^4u^i}\\
&+h \sum_{e\in g\circ S^1}\ltwoe{\jump{D^3 u^i}} \\
& + h
\sum_{e\in g\circ S^1}\ltwoe{\jump{D^5 u^i}}  + 2 \ltwonormh{D^4u^i}\Big)\\
& + \sum_{e\in g\circ S^1}\ltwoe{\jump{D^3 u^i
-P_{\mathcal{K}_e}\big[D^3 u^i\big]}}\\
& +h^s |u|_{m+s,\Omegah},
\end{align*}
for any $s\geq 0$, so the result follows.
\end{proof}

\begin{remark}[Conforming reconstruction operator]
In the analysis above we make heavy use of the projection operator
$P_{\mathcal{K}_e}$ as it is constructed in such a way as to be continuous over the
edge $e$ (trivially so as the range of this operator is the space of constants).
\cite{georgoulis2017analysis} proposes a $\mathsf{C}^1$ reconstruction operator
with attractive properties (see e.g. equation (16) in the reference) that could
be explored instead. These properties concern error estimation of the projection
error on elements $K$ of the triangulation. However, convergence in terms of $h$
is not readily obtained using such operators as the preceding analysis deals
with terms of like $\sum_{e\in g\circ S^1}\ltwoe{\jump{D^3 u^i
-P_{\mathcal{K}_e}\big[D^3 u^i\big]}}$ concerning the jump over edges. By
using a trace inequality to pull the estimate to $K$ we lose an order of
differentiation which is reflected in the estimation properties of the
aforementioned operator (roughly speaking: the $L^2$ norm on $\partial K$ can be
bounded by the $H^1$ norm on $K$). It is therefore not certain that such a
conforming reconstruction operator yields estimates that scale well with $h$.
\end{remark}

In summary, the finite element discretisation in \eqref{Ham:disc} is convergent for $\tilde f \in L^2(S^1)^d$ and mesh-aligned curve $g\in\mathsf{C}^0(S^1)^d$.
We saw that the presence of the singular curve integral as a source term prohibits global regularity, and that despite recovering almost everywhere $\mathsf{C}^\infty$ regularity we cannot attain linear convergence as a function of $h$ using the conforming relatives used in  corollary \ref{cor:strang}.

\section{Summary}\label{conclusion}

We have presented error analysis for a higher-order nonconforming finite element method with singular data. In doing so we have confirmed a standard result: nonconforming finite elements are well-suited, in terms of proving convergence results for the discretisation, to problems with smooth data. The lack of at least global $L^2$ regularity of the source terms considered here prohibits the discovery of convergence rate across the singularity. We therefore emphasise the need for higher-order conforming finite element methods as this would greatly simplify the analysis carried out in section \ref{sec:trilap}.


\addcontentsline{toc}{section}{References}

\printbibliography

\end{document}